\documentclass[12pt,a4paper, reqno]{amsart}
\usepackage{amsfonts,amsmath,amssymb}
\usepackage{hyperref}
\usepackage{latexsym,fullpage,amsfonts,amssymb,amsmath,amscd,graphics,epic,enumerate}
\usepackage[all]{xy}
\usepackage{amssymb,amsthm,amsxtra,comment}
\usepackage{color}
\usepackage[dvipsnames]{xcolor}
\usepackage{amscd}
\usepackage{amsthm}
\usepackage{amsfonts,cancel}
\usepackage{amssymb}
\usepackage{url}
\usepackage{bbm}
\usepackage{wasysym}
\usepackage{MnSymbol}
\usepackage{makecell}
\usepackage{tikz}

\newif\ifgr
 
 \theoremstyle{plain}
\newtheorem{theorem}{Theorem}[section]
\newtheorem{lemma}[theorem]{Lemma}

\newtheorem{remark}[theorem]{Remark}

\theoremstyle{definition}

\theoremstyle{remark}

 \newcommand{\op}{\operatorname}

\newcommand{\Hom}{\operatorname{Hom}}

\newcommand{\sch}{\operatorname{sch}}

\newcommand{\Coind}{\operatorname{Coind}}

\newcommand{\thin}{\operatorname{thin}}
\newcommand{\thick}{\operatorname{thick}}

\newcommand{\sgn}{\operatorname{sgn}}

\renewcommand{\gg}{\mathfrak{g}}

\renewcommand{\dim}{\mathrm{dim}}

\newcommand{\supp}{\mathrm{supp}}
\newcommand{\Id}{\mathrm{Id}}

\newcommand{\black}{\color{black}}

\newcommand{\nc}{\newcommand}
\nc{\vareps}{\varepsilon}
\nc{\varesp}{\varepsilon}
\nc{\veps}{\varepsilon}

 \def\fp{\mathfrak{p}}
 
 \def\<{\langle}
  \def\>{\rangle}

\DeclareMathOperator{\tr}{tr}

\newcommand{\End}{\mathrm{End}}

\usepackage[colorinlistoftodos]{todonotes}
\allowdisplaybreaks 
\begin{document}

\title{Denominator identities for the periplectic Lie superalgebra}

\author{Crystal Hoyt
\and Mee Seong Im    
\and Shifra Reif} 
\address{Crystal Hoyt, Department of Mathematics,  ORT Braude College \& Weizmann Institute, Israel.
{\it E-mail:} \tt{crystal.hoyt@weizmann.ac.il}}
\address{Mee Seong Im, Department of Mathematical Sciences, United States Military Academy, West Point, NY 10996 USA.
{\it E-mail:} \tt{meeseongim@gmail.com}}
\address{Shifra Reif, Department of Mathematics, Bar-Ilan University, Ramat Gan, Israel.
{\it E-mail:} \tt{shifra.reif@biu.ac.il}}
\date{\today}
\thanks{The first author has been partially funded by ORT Braude College's Research Authority. The second author acknowledges Bar-Ilan University at Ramat Gan, Israel for excellent working conditions. The second and third authors are partially funded by ISF Grant No. 1221/17.} 

\begin{abstract}  
We prove denominator identities for the periplectic Lie superalgebra $\mathfrak{p}(n)$, thereby completing the problem of finding denominator identities for all simple classical finite-dimensional Lie superalgebras. \\
\noindent\textbf{Mathematics subject classification (2010):} Primary 
17B20  
17B10, 
05E10,  
05E05.  
\\
\noindent\textbf{Key words:} periplectic Lie superalgebra, Weyl character formula, denominator identity,  Kac--Wakimoto character formula, supercharacters, Kac modules.
\end{abstract}

\maketitle
\bibliographystyle{amsalpha}  
\setcounter{tocdepth}{3}

\section{Introduction}\label{sec:intro}

The classical Weyl character formula describes the character of a simple finite-dimensional module over a Lie algebra in terms of characters of modules which are easy to compute. This formula admits generalizations to infinite-dimensional Lie algebras as well as to some Lie superalgebras. 
It turns out that a particularly interesting case is when the formula is specialized to the trivial representation. In this case, one side of the equality is easy to understand, namely it is 1, and since the other side is a fraction, the resulting identity is called the \emph{denominator identity}.

Denominator identities  for Lie algebras  have numerous applications to algebra,  combinatorics and number theory. 
For example, the denominator identities for affine Lie algebras \cite{Kac74} \black turned out to be the famous Macdonald identities \cite{MR0357528}. The simplest case of this is the Jacobi triple product identity, which is the Macdonald identity for the affine root system of type $A_1$ and is the denominator identity for the affine Lie algebra $\widehat{\mathfrak{sl}(2)}$.

Denominator identities for Lie superalgebras also have important applications. For example, V.~Kac and M.~Wakimoto showed that a specialization of the affine denominator give formulas for computing the number of ways to decompose an integer as a sum of $k$-squares or as a sum of triangular numbers for certain values of $k$ in \cite[Section 5]{MR1327543}.
In addition, super denominator identities were used to determine the simplicity conditions of various $W$-algebras in  \cite{MR2323540,MR2512631s} and to recover the Theta correspondence for compact dual pairs of real Lie groups in \cite{MR2902248}.

The denominator identity for simple Lie algebras and symmetrizable Kac--Moody algebras takes the form
$$e^{\rho} R= \sum_{w\in W}(\sgn w)w\left( e^{\rho}\right),$$where $R$ is the denominator, $W$ is the Weyl group and $\rho$ is a certain element in the dual of the Cartan subalgebra  \cite{Kac74}.
For Lie superalgebras, one needs another ingredient, namely the notion of a maximal isotropic set of roots, which was introduced by V.~Kac and M.~Wakimoto in \cite{MR1327543}  and now plays a key role in character formulas for Lie superalgebras. 

Denominator identities for simple basic classical Lie superalgebras and queer Lie superalgebras, and  for  their non-twisted affinizations were formulated by V.~Kac and M.~Wakimoto in \cite{MR1327543}, where they proved the defect one case.  The proofs of these denominator identities were completed by M.~Gorelik and D.~Zagier in \cite{Gorelik-Maria-2010,MR2796063,MR2866851,MR1809285}, and generalizations appeared in \cite{gorelik2012denominator,MR3244923,MR2902248}.

The only  classical  finite-dimensional Lie superalgebra that remained was the periplectic Lie superalgebra $\mathfrak{p}(n)$. This algebra first appeared in V.~Kac's well-known classification of simple finite-dimensional Lie superalgebras \cite{MR0486011}. The representation theory of $\mathfrak{p}(n)$ has been studied by V.~Serganova in \cite{Ser02} and others in \cite{Moo03,Che15}; however, one difficulty arises owing to the lack of a quadratic Casimir element (see \cite{G01} for a description of the center of the universal enveloping algebra of $\mathfrak{p}(n)$).  Recently, a large breakthrough  in the representation theory of $\mathfrak p(n)$ was accomplished after the introduction of a ``fake Casimir'', a.k.a. tensor Casimir \cite{BDEHHILNSS}. This advancement has promoted a resurgence of interest in the Lie superalgebra $\mathfrak{p}(n)$; see for example \cite{CC,EAS18,EAS19,IRS19}.

In this paper, we state and  prove two different denominator identities for the periplectic Lie superalgebra, namely for two nonconjugate Borel subalgebras $\mathfrak{b}^{\thick}$ and $\mathfrak{b}^{\thin}$ (see Section~\ref{sec rs}).  These identities take the form
$$
e^{\rho} R
= \sum_{w\in W} (\sgn w) w
 \left(  
\frac{e^{\rho}}{(1-e^{-\beta_1})(1-e^{-\beta_1-\beta_2})\cdots (1-e^{-\beta_1-\ldots -\beta_{r}})} 
\right),
	$$
where  $\{\beta_1,\ldots,\beta_r\}$ is an explicitly defined maximal set of mutually orthogonal odd roots. These denominator identities are similar to the identities  given in \cite[Thm. 1.1]{papi2010denominator} for basic Lie superalgebras.
For the Borel subalgebra $\mathfrak{b}^{\thin}$, we also have an identity in a form similar to \cite[Thm 2.1]{MR1327543}.
We note that other Borel subalgebras of $\mathfrak{p}(n)$ cannot admit a denominator identity in the classical form, as $e^{\rho} R$ is not $W$-anti-invariant in these other cases (see Remark~\ref{other borels}). 

The proofs given in this paper are of a combinatorial nature and do not rely on deep theorems. Note that the standard proof techniques used for basic classical Lie superalgebras do not apply here as the periplectic Lie superalgebra lacks an invariant bilinear form and a  Casimir element.

We hope that these identities will set the ground to prove a general character formula for all finite-dimensional simple modules over $\mathfrak p(n)$.\\

Our paper is organized as follows. Section~\ref{sec prelim} contains preliminary definitions and lemmas. In Sections~\ref{sec thin} and \ref{sec thick}, we state and prove the thin and thick denominator identities, respectively. The final section of our paper contains a remark on character formulas
and an open problem concerning the homological complex behind these denominator identities.

\subsection{Acknowledgments} The authors would like to thank Maria Gorelik and Inna Entova-Aizenbud for useful discussions and helpful suggestions.


\section{The periplectic Lie superalgebra}\label{sec prelim}

    \subsection{Lie superalgebras}

    Let $\mathfrak{gl}(n|n)$, the general linear Lie superalgebra over $\mathbb{C}$, and 
    let $V = V_{\bar 0}\oplus V_{\bar 1}$ be a $\mathbb{Z}_2$-graded vector superspace.  
    The parity of a homogeneous vector $v\in V_{\bar 0}$ is defined as 
    $\bar v=\bar 0\in \mathbb{Z}_2 =\{ \bar 0 , \bar 1 \}$, 
    while the parity of an odd vector $v\in V_{\bar 1}$ is defined as $\bar v = \bar 1$. 
    If the parity of a vector $v$ is $\bar 0$ or $\bar 1$, 
    then $v$ has degree $0$ or $1$, respectively.  
    If the notation $\bar v$ appears in formulas, we will assume that $v$ is homogeneous. 

    The general linear Lie superalgebra may be identified with the endomorphism algebra 
    $\End(V_{\bar{0}}\oplus V_{\bar 1})$, where 
    $\dim \: V_{\bar 0} = \dim \: V_{\bar 1}=n$. Then $\mathfrak{gl}(n|n) 
    = \mathfrak{gl}(n|n)_{\bar 0}\oplus \mathfrak{gl}(n|n)_{\bar 1}$, 
    where 
    \[ 
    \mathfrak{gl}(n|n)_{\bar 0} = \End(V_{\bar 0}) \oplus \End(V_{\bar 1}) 
    \qquad 
    \mbox{ and } 
    \qquad 
    \mathfrak{gl}(n|n)_{\bar 1} = \Hom(V_{\bar 0}, V_{\bar 1}) \oplus \Hom(V_{\bar 1},V_{\bar 0}). 
    \] 
    So $\mathfrak{gl}(n|n)_{\bar 0}$ consists of parity-preserving linear maps while $\mathfrak{gl}(n|n)_{\bar 1}$ consists of parity-switching maps. 
    We also have a bilinear operation on $\mathfrak{gl}(n|n)$: 
    \[
    [x,y] = xy-(-1)^{\bar x \bar y}yx 
    \] 
    on homogeneous elements, which then extends linearly to all of $\mathfrak{gl}(n|n)$. 
    By fixing a basis of $V_{\bar 0}$ and $V_{\bar 1}$, the Lie superalgebra $\mathfrak{gl}(n|n)$ can be realized as the set of $2n\times 2n$ matrices, where 
    \[ 
    \mathfrak{gl}(n|n)_{\bar 0} = 
    \left\{  
    \begin{pmatrix}
    A & 0 \\ 
    0 & B \\ 
    \end{pmatrix} : 
    A, B\in M_{n,n}
    \right\}
    \quad 
    \mbox{ and }
    \quad 
    \mathfrak{gl}(n|n)_{\bar 1} = 
    \left\{  
    \begin{pmatrix} 
    0 & C \\ 
    D & 0 \\ 
    \end{pmatrix} : 
    C,D \in M_{n,n}
    \right\}, 
    \] 
    where $M_{n,n}$ are $n\times n$ complex matrices. 

    Let $V$ be an $(n|n)$-dimensional vector superspace equipped with a non-degenerate odd symmetric form
    \begin{eqnarray}
    \label{beta}
    \beta:V\otimes V\to\mathbb C,\quad \beta(v,w)=\beta(w,v), \quad\text{and}\quad \beta(v,w)=0 \quad \text{if} \quad  \bar{v}=\bar{w}.
    \end{eqnarray}
    Then $\op{End}_{\mathbb C}(V)$ inherits the structure of a vector superspace from $V$. 
    
    The periplectic superalgebra $\mathfrak{p}(n)$ is defined to be the Lie superalgebra of all $X\in\operatorname{End}_{\mathbb C}(V)$ preserving $\beta$, i.e., $\beta$ satisfies the condition  $$\beta(Xv,w)+(-1)^{\bar{X}\bar{v}}\beta(v,Xw)=0.$$

    \begin{remark}\label{rmk:basis}
    With respect to a fixed bases for $V$, the matrix of $X\in \mathfrak{p}(n)$ has the form 
    $\left(\begin{smallmatrix}
    A&B\\
    C&-A^t
    \end{smallmatrix}
    \right)$, 
    where $A$, $B$, $C$ are $n\times n$ matrices such that $B^t=B$ and $C^t=-C$. 
     Note that $\fp(n)$ is not itself simple; however, the subalgebra  $\mathfrak{sp}(n)$ obtained by imposing the additional condition $\tr A = 0$ is simple and has codimension $1$ in $\mathfrak{p}(n)$.
       \end{remark}

     \subsection{Root systems}\label{sec rs}
    \label{subsection:root-systems}
  From this point on, we will let $\mathfrak{g}:= \mathfrak{p}(n)$. Fix the Cartan subalgebra $\mathfrak{h}$ of $\mathfrak g$ which consists of diagonal matrices and let  $\{ \varepsilon_1,\ldots, \varepsilon_n \}$ be the standard basis of $\mathfrak h^*$.  Note that $\mathfrak{h} \subset \mathfrak{g}_{\bar{0}}$. We have a root space decomposition 
  $\mathfrak{g} 
= \mathfrak{h}\oplus \left( \bigoplus_{\alpha\in \Delta} \mathfrak{g}_{\alpha}\right)$, where $\Delta$ denotes the set of roots of $\gg$. The set $\Delta$ decomposes as
$$
\Delta=
 \Delta(\mathfrak{g}_{-1})\cup \Delta(\mathfrak{g}_0)\cup \Delta(\mathfrak{g}_1), 
$$
 where 
\begin{align*} 
\Delta(\mathfrak{g}_0) &= \{ 
\varepsilon_i - \varepsilon_j: 1\leq i \not= j \leq n
\},   \\
\Delta(\mathfrak{g}_1) = 
\{ 
\varepsilon_i + \varepsilon_j : 1\leq i\le j &\leq n 
\} \quad 
\mbox{ and }
\quad 
\Delta(\mathfrak{g}_{-1}) = 
\{ 
-(\varepsilon_i + \varepsilon_j) : 1\leq i< j\leq n 
\}, 
\end{align*}
and moreover, $\mathfrak{g}$ has a (short) $\mathbb Z$-grading $\mathfrak g=\mathfrak g_{-1}\oplus\mathfrak g_0\oplus\mathfrak g_1$, where  $\mathfrak g_k:=\bigoplus_{\alpha\in \Delta(\mathfrak{g}_{k})} \mathfrak{g}_{\alpha}$.  This $\mathbb Z$-grading is compatible with the Lie superalgebra structure on $\mathfrak g$ as $\mathfrak g_{\bar 0}=\mathfrak g_0$ and  $\mathfrak g_{\bar 1}=\mathfrak g_{-1}\oplus\mathfrak g_1$.

Additionally, $\Delta$ decomposes into even and odd roots $\Delta=\Delta_{\bar 0}\cup\Delta_{\bar 1}$, where $\Delta_{\bar 0} = \Delta_{0}$ and $\Delta_{\bar 1} = \Delta_{-1}\cup\Delta_{ 1}$. 
We can choose a set of positive roots $\Delta^+\subset\Delta$ and consider the corresponding Borel subalgebra $\mathfrak b=\mathfrak{h}\oplus \left( \bigoplus_{\alpha\in \Delta^+} \mathfrak{g}_{\alpha}\right)$. In what follows, we will always assume that 
$$\Delta^+_{\bar 0} = \{ \varepsilon_i-\varepsilon_j:1\leq  i < j\leq n \}$$
and that $\Delta_{\bar 1}^+$ is either $\Delta(\mathfrak{g}_1)$  or $\Delta(\mathfrak{g}_{-1})$. We denote the corresponding Borel subalgebras of  $\mathfrak g$ by $\mathfrak{b}^{\thick}:=\mathfrak{b}_0\oplus \mathfrak{g}_1$ and $\mathfrak{b}^{\thin}:=\mathfrak{b}_0\oplus \mathfrak{g}_{-1}$, respectively. 
Let $\rho_{\bar{0}}:=\frac{1}{2}\left(\sum_{\alpha\in\Delta_{\bar{0}}^{+}}\alpha\right)$, 
	 $\rho_{\bar{1}}:=\frac{1}{2}\left(\sum_{\alpha\in\Delta_{\bar{1}}^{+}}\alpha\right)$ and
	$\rho=\rho_{n}:=\rho_{\bar{0}}-\rho_{\bar{1}}$.

We will work over the rational function field generated by $e^\lambda$, where $\lambda\in\mathfrak{h}^*$. 
Let 
\[ 
R_{0,n} = \prod_{\alpha\in\Delta^+(\mathfrak{g}_{0})} (1-e^{-\alpha}), 
\qquad 
R_{1,n}=\prod_{\alpha\in\Delta(\mathfrak{g}_{1})} (1-e^{-\alpha}), 
\qquad 
R_{-1,n}=\prod_{\alpha\in\Delta(\mathfrak{g}_{-1})} (1-e^{-\alpha}). 
\] 
We will write ${R}_0$, ${R}_{1}$ ${R}_{-1}$, 
respectively, when it is clear what the algebra is. We let $R=\frac{R_0}{R_{\bar 1}}$, where $R_{\bar 1}=R_{1}$ or $R_{-1}$.

\begin{remark}
The inverse of $e^{2\rho} R$ is the supercharacter of $U(\mathfrak{n})$ (up to a sign), where $\mathfrak{n}$ denotes the nilradical of $\mathfrak{b}$ (here $\mathfrak{b}=\mathfrak{b}^{\thick}$ or $\mathfrak{b}^{\thick}$ for $R_{\bar 1}=R_{1}$ or $R_{-1}$, respectively). One can also consider character versions of the denominator identity. See Section~\ref{sec char}.
\end{remark}

A polynomial $f$ is anti-invariant or skew-invariant if 
$w.f=(\sgn w)f$ for any $w\in W$. 
Note that ${R}_{-1}$, ${R}_{1}$ are  $W$-invariant and  $e^{\rho}{R}_0$ is $W$-anti-invariant. 

\begin{remark} \label{other borels} 
The sets $\Delta(\mathfrak g_1)$ and $\Delta(\mathfrak g_{-1})$ are the only choices of positive odd roots for which $e^\rho R$ is $W$-anti-invariant with respect to our fixed choice of positive even roots. Thus,  there does not exist a denominator identity for other  Borel subalgebras, since $e^\rho R$ can not equal an alternating sum over the Weyl group. 
\end{remark}

Define 
$\mathcal{F}_W(a) := \sum_{w\in W} (\sgn w) w(a)$.  
Let $y=\sum_{\mu} a_{\mu}e^{\mu}$, where $a_{\mu}\in \mathbb{Q}$. The support of $y$ is defined to be 
\[
\supp(y) =\{ \mu: a_{\mu}\not=0\}. 
\] 
An element $\lambda=\sum a_i \varepsilon_i$ is regular if and only if it has a distinct coefficient for
 	every  $\varepsilon_{i}$. 
That is, $\lambda$ is regular if $w(\lambda) = \lambda$ implies $w=\Id$. 
Moreover, its orbit is of maximal size. 

We cite \cite[Lemma 4.1.1 (ii)]{MR2866851}:  
\begin{lemma}
\label{lem:stab-refl}
For any 
$\mu\in\mathfrak{h}_{\mathbb{R}}^{*}$, 
the stabilizer of $\mu$ in $W$ is either trivial or contains a reflection.
\end{lemma}
This implies the stabilizer of a non-regular point $\mu$ in $W$ contains a reflection. Thus the space of $W$-anti-invariant elements is spanned by $\mathcal F_{W}\left( e^\mu\right)$, where $\mu$ is regular.

We call the orbit $W(\mu)$ regular if $\mu$ is regular; thus, regular orbits consist of regular points.
\begin{lemma}\label{orbit}
The  support of a $W$-anti-invariant element is a union of regular $W$-orbits. 
\end{lemma}
 
\begin{proof}   
Since $\mathcal F_{W}\left( e^\lambda\right)=0$ for non-regular $\lambda $, only regular elements appear in the support of $\mathcal F_{W}(a)$.
\end{proof}

 
\section{Thin denominator identity for $\mathfrak{p}(n)$}\label{sec thin}
   
 In this section, we present denominator identities for the Borel subalgebra $\mathfrak{b}^{\thin}$ of $\fp(n)$, namely when 
    $\Delta_{\bar{1}}^{+}=\Delta(\mathfrak{g}_{-1})$. In this case, $\rho_{\bar{1}}= \left(\frac{1-n}{2}\right)\sum_{i=1}^{n} \varepsilon_i $ and
    $
    \rho 
    = \sum_{i=1}^{n} (n-i)\varepsilon_i 
    $. 
    
    Let
	$R=\frac{R_0}{R_{-1}}$, where $R_{-1}=\prod_{1\le i<j\le n}\left(1-e^{\varepsilon_i+\varepsilon_j}\right)$.
        Set $r=\left\lfloor\frac{n}{2}\right\rfloor$     and let    
        $$
        \beta_1 = -(\varepsilon_1+\varepsilon_2), 
        \hspace{3mm}
        \beta_2 = -(\varepsilon_3+\varepsilon_{4}),   	\ldots, 	
        \beta_{r} =-(\varepsilon_{2r-1}+\varepsilon_{2r}).
        $$
         We define 
$$\rho^{\Uparrow}:=\rho+(n-2)\beta_1+(n-4)\beta_2+\ldots+(n-2r)\beta_{r}    =\varepsilon_1+\varepsilon_3+\ldots+\varepsilon_{2r-1}.
   $$

Here is one form of the thin denominator identity. 
	\begin{theorem}\label{th1} Let $\mathfrak g=\fp(n)$ and $\Delta_{\bar{1}}^{+}=\Delta(\mathfrak{g}_{-1})$. 
	Then  
	$$	
	e^{\rho}R=\frac{1}{r !}
	     \sum_{w\in W} (\sgn w) w\left(\frac{e^{\rho^{\Uparrow}} }{
	\prod_{\beta\in S} \left(1- e^{-\beta} \right)} 
	   \right),
	$$
	where $S=\{\beta_1,\beta_2,\ldots,\beta_r\}$.
	\end{theorem}

	\begin{proof}
	By applying the permutation $\tau_n:=(2t-1\to t; 2t\to n+1-t)$ to the RHS we obtain the equivalent expression
	\begin{equation}
	\label{eqn:perm}
	e^{\rho}R= (\sgn \tau_n)\frac{1}{r !}
	     \sum_{w\in W} (\sgn w) w\left(\frac{e^{\vareps_1+\vareps_2+\ldots+\vareps_{r}} }{
	\prod_{\beta'\in S'} \left(1- e^{-\beta'} \right)} 
	   \right),
	\end{equation}
	where $S'=\{-(\varepsilon_1+\varepsilon_n),-(\varepsilon_2+\varepsilon_{n-1}),\ldots,-(\varepsilon_{r}+\varepsilon_{n+1-r})\}$.
	
	Since $R_{-1}$ is $W$-invariant, equation \eqref{eqn:perm} is equivalent to
		\begin{equation}\label{eqn 001}
     (\sgn \tau_n){r !}e^{\rho}R_{0}=
     \mathcal F_W \left(e^{\vareps_1+\vareps_2+\ldots+\vareps_{r}} 
	\prod_{\alpha\in\Delta(\mathfrak g_{-1})\setminus S' } 
	\left(1- e^{-\alpha} \right)   \right).
	\end{equation}
	Since $\rho_{\bar{1}}$ is $W$-invariant,  the Weyl denominator identity for $\mathfrak{sl}(n)$ 
	yields the equivalent expression
		\begin{equation}\label{eqn U}  
		(\sgn\tau_n){r !}\mathcal F_W \left( e^{\rho}\right) =
		\mathcal F_W \left(e^{\vareps_1+\vareps_2+\ldots+\vareps_{r}} 
	\prod_{(i,j)\in U} (1-e^{\vareps_i+\vareps_j})\right),
		\end{equation}
	where
	\[
	U := \{(i,j)\in\mathbb{Z}_{>0}\times\mathbb{Z}_{>0}\mid  1\leq i<j\leq n,\ i+j\not=n+1\}.
	\]
	Note that the Weyl denominator identity for the Lie algebra $\mathfrak{g}_{\bar{0}}$ gives $e^{\rho_{\bar{0}}}R_0=\mathcal{F}_W 
	\left( e^{\rho_{\bar{0}}}\right)$. 
	So we can multiply both sides by $e^{-\rho_{\bar{1}}}$ and pass it through  $\mathcal F_W$ since it is invariant.
	We also note that both sides of 
	\eqref{eqn U} are $W$-skew-invariant.
	Let
	\begin{equation}\label{def A} 
	\mathcal A:=e^{\vareps_1+\vareps_2+\ldots+\vareps_{r}} 
	\prod_{(i,j)\in U} (1-e^{\vareps_i+\vareps_j})
	= \sum_{\nu} a_{\nu}e^{\nu}.
	\end{equation}
	Then
	\begin{equation}\label{supp A}
	\supp (\mathcal A)\subset \left\{\sum_{i=1}^n k_i\vareps_i: 0\leq k_i\leq n-1,\ 1\leq k_{1},\ldots,k_{r} 
	\leq n-1,\  0\leq k_{n+1-r},\ldots,k_n\leq n-2\right\},
	\end{equation}
	since 
    $\mathcal A$ can be expressed using $k_i$, i.e., 
    we get the bounds on the coefficients $k_i$ 
    by counting the elements in the set $U$, 
	and the regular elements in $\supp (\mathcal A)$ lie in the orbit $W(\rho)$ 
	since elements in the orbits of $W(\rho)$ have coefficients $0,1,\ldots,n-1$ (in some order). 
	Now since $\mathcal F_{W}\left( e^\mu\right)=0$ for non-regular elements $\mu\in\mathfrak h^* $, 
	we have that the RHS of \eqref{eqn U} equals
	$$
	\mathcal F_{W}\left(\mathcal A\right)=\mathcal F_W\left( \sum_{y\in W} a_{y\rho} e^{y\rho}\right)
	= j\ \mathcal F_{W}\left( e^{\rho}\right),
	$$
	where 
	$$
	j:=\sum_{y\in W} (\sgn y) \ a_{y\rho},
	$$
	by switching the sums and then changing the indexing set.  
	We get a summation of the form 
	$\mathcal{F}_W\left(  e^{y\rho}\right)$, and we reindex since the sum is over all of $W$. 
	Hence to prove \eqref{eqn U} and deduce the theorem, it remains to show that $j=(\sgn\tau_n) r!$.

	Consider the following embedding $\iota: S_{r}\to S_n$: each permutation $\sigma\in S_{r}$
	maps to the corresponding permutation of the set $1,2,\ldots,r$ and
	the corresponding permutation of the set $n,n-1,\ldots,n+1-r$ 
	(for instance, for $n=5$ we have 
	$\iota((12))= (12)(54)$); 
	note that $\iota(S_{r})$ consists of even permutations. Clearly, $\mathcal A$ is $\iota(S_{r})$-invariant, as 
	$$\mathcal A=\frac{ e^{\vareps_1+\vareps_2+\ldots+\vareps_{r}}R_{-1}}{
	\prod_{\beta'\in S'} \left(1- e^{-\beta'
	}\right)}.$$ 
	So by fixing a set of representatives of the left cosets of $\iota(S_{r})$ in $S_n$
	to be 
	$$ 
	S_n/S_{r} := \{\sigma\in S_n :  \sigma(n -(r-1))< \ldots <\sigma(n-1)< \sigma(n)\},
	$$
	we have
	$$ 
	j=r !  \sum_{y\in S_n/S_{r}} (\sgn y)  a_{y\rho}.
	$$

	We derive from \eqref{def A} that
	$$
	\mathcal A=\sum_{P\subset U} (-1)^{|P|} e^{\vareps_1+\vareps_2+\ldots+\vareps_{r}+\sum_{(i,j)\in P}(\varepsilon_i+\varepsilon_j)},
	$$
	where each $P$ is a subset of $U$, and $|P|$ denotes the number of elements in $P$.
		 Suppose
	$$
	\vareps_1+\vareps_2+\ldots+\vareps_{r}+\sum_{(i,j)\in P'}(\varepsilon_i+\varepsilon_j)
	= \sum_{i=1}^n k_i\vareps_i=y'\rho
	$$
	for some $y'\in S_n/S_{r}$ and $P'\subset U$. 
		We will prove that necessarily $y'=\Id$ and
		 $$P' = 
		 \{(i,j)\in\mathbb{Z}_{>0}\times\mathbb{Z}_{>0} :  i<j,\ i+j\leq n\}.$$ 
		First, note that
	$$
	\{k_1,\ldots,k_n\}=\{0,1,\ldots,n-1\}\ \text{ and } \ \ k_{n+1-r}>\ldots >k_{n-1}> k_n 
	$$
	since $\rho = \sum_{i=1}^{n} (n-i)\varepsilon_i$ and $y'\in S_n/S_{r}$. Also, recall the conditions on $\supp (\mathcal A)$ given in \eqref{supp A}. 
	
           We will prove that $k_i=n-i$ for all $i=1,\ldots,n$. 
        Our base case is to show that $k_1=n-1$, $k_n=0$, and that $ (1,i)\in P'\ \Leftrightarrow\ i\not=1,n$.
    Now since $k_i\geq 1$ for $i\leq r$ and $k_n<k_i$ for $i\in\{n+1-r,\ldots,n-1\}$, 
	we can have either $k_n=0$ or $k_{r+1}=0$ (in the case that $n=2r+1$). 
	However, if $k_{r+1}=0$ we reach a contradiction that $(n-1)$ 
	could not occur as a coefficient. 
	Indeed, suppose $n=2r+1$ and $k_{r+1}=0$ and take $j$ 
	such that $k_j=n-1$. Then $j\leq r$ and the elements $(j,j), (j,r+1),(j,n+1-j)$ are not in $ P'$. Since $k_j=n-1$, two of these pairs must coincide, implying  $j=r+1$, which is a contradiction.
	Hence, $k_n=0$ and so $(i,n)\not\in P'$ for all $i$. Take $j$ such that  $k_j=n-1$; 
	then
	$j\leq r$ and $(j,i)\in P'$ for all $i\not=j,n+1-j,n$. Thus
	$n+1-j=n$, that is $j=1$. Therefore, we obtain $k_1=n-1$, $k_n=0$. It follows that $ (1,i)\in P'$ if and only if $i\not=1,n$.

	Let $t\leq r$, and suppose for the induction hypothesis that for all $i=1,\ldots,t-1$ we have: $k_i=n-i$, $k_{n+1-i}=i-1$, and 
	\begin{equation}\label{eq ind}
	(i,j)\in P'  \ \Leftrightarrow\ i\leq\min\{j-1,n-j\}.\end{equation} 
	One can prove that the induction hypothesis implies that $k_i\geq t$ for all $i\leq r$, and that $k_i< n-t$ for $i> r$.
	Suppose $k_p=n-t$ and $k_q=t-1$. Then $t\leq p\leq r$ and $r<  q\leq n+1-t$.  It follows (indirectly) from the induction hypothesis that  $$(p,p),(p,n+1-p),(p,q),(p,n),(p,n-1),(p,n-2),\ldots,(p,n+2-t)\not\in P'.$$  This implies
	that $p+q=n+1$. It follows that $k_t=n-t$ since $p\neq q$ and $k_t>k_i$ for all $r<i<t$. Hence, $k_{n+1-t}=t-1$. Finally, since the elements $(t,t),(t,n),(t,n-1),\ldots,(t,n+1-t)$ are not in $P'$ and yet $k_t=n-t$, we see that condition \eqref{eq ind} also holds for $i=t$. This concludes the induction proof. Hence,
	$k_i=n-i$ for each $i$ and  $y'=\Id$. Therefore,
	 $a_{y\rho}=0$ for $y\in S_n/S_r$ such that 	$y\not=\Id$. 
	
	Next we will prove that $a_{\rho}=\sgn\tau_n$.
		 For this, we need to show that $\sgn\tau_n=(-1)^{|P'|}$, where $$\vareps_1+\vareps_2+\ldots+\vareps_{r}+\sum_{(i,j)\in P'}(\varepsilon_i+\varepsilon_j)=\rho.$$
    	We claim that
 $\sgn\tau_n=1$ if $n$ is even, and $\sgn\tau_n=(-1)^k$ for 
$n=2k+1$. Indeed, one can check directly that $\tau_1,\tau_2=\Id$,
$$\tau_{2k+1}=(k+1,k+2,\ldots,2k+1)\tau_{2k},\ \ \tau_{2k}=(k+1,k+2,\ldots,2k)\tau_{2k-1}.$$
    Thus,  $\sgn\tau_{2k+2}=\sgn\tau_{2k}=1$, while
 $\sgn\tau_{2k+1}=(-1)(\sgn\tau_{2k-1})=(-1)^{k}$, and the claim follows by induction.
	On the other hand, counting coefficients for $\rho-(\vareps_1+\vareps_3+\ldots+\vareps_{2r-1})$ yields
	$$|P'|=\frac{1}{2}\left(\frac{n(n-1)}{2}-\left\lfloor\frac{n}{2}\right\rfloor\right).$$ 
	If $n$ is even then $|P'|=\frac{n(n-2)}{4}$, which is always even. If $n$ is odd then $|P'|=\left(\frac{n-1}{2}\right)^2$, which is even precisely when $n=4k+1$ for some $k\in\mathbb N$.
	Thus $a_{\rho}=(-1)^{|P'|}=\sgn\tau$.
	
	Therefore, $j=(\sgn\tau) r!$, and the theorem follows.
\end{proof}

Here is another form of the thin denominator identity.

\begin{theorem} 
\label{thm thin sum of betas} Let $\mathfrak g=\fp(n)$ and $\Delta_{\bar{1}}^{+}=\Delta(\mathfrak{g}_{-1})$. 
	Then  
\[ 
e^{\rho} R
= \sum_{w\in W} (\sgn w) w
 \left(  
\frac{e^{\rho}}{(1-e^{-\beta_1})(1-e^{-\beta_1-\beta_2})\cdots (1-e^{-\beta_1-\ldots -\beta_{r}})} 
\right),
\] 
where $r=\lfloor n/2\rfloor$ and 
$\beta_1 = -(\varepsilon_1+\varepsilon_2), 
\beta_2 = -(\varepsilon_3+\varepsilon_4),   
\ldots, 
\beta_{r} =-(\varepsilon_{2r-1}+\varepsilon_{2r})$.

\end{theorem}

\begin{proof}
For $\mu\in \mathfrak{h}^*$, write 
$$
X_\mu:=\mathcal{F}_W
\left( \frac{e^{\mu}}{
(1-e^{-\beta_1})(1-e^{-\beta_1-\beta_2})\cdots 
(1-e^{-\beta_1-\ldots-\beta_r})
} \right).
$$
We show that 
\begin{equation} 
X_{\rho} = \frac{1}{r!} 
\mathcal{F}_W 
\left(  
\frac{e^{\rho^{\Uparrow}}}{ 
(1-e^{-\beta_1})(1-e^{- \beta_2})\cdots (1-e^{-\beta_r})  
}
\right),
\label{from product to sums of betas form}
\end{equation}
where
$ 
\rho^{\Uparrow}=\varepsilon_1+\varepsilon_3+\ldots+\varepsilon_{2r-1}.
   $

 We first claim that 
$
 \label{F_rho=F_rho^uparrow}
X_\rho=X_{\rho^\Uparrow}.
$
 We expand $X_\rho$ and $X_{\rho^\Uparrow}$ 
 as a geometric series in the domain $\left|e^\alpha \right|<1$ 
 for $\alpha>0$. 
 Note that $w\beta_i>0$ 
 for every $w\in W$. Thus,
  \begin{equation*}
\begin{split}
    X_\rho &=\sum_{m_1\ge m_2\ge \ldots\ge m_r\ge 0}\mathcal{F}_W \left(e^{\rho-m_1\beta_1-\ldots-m_r\beta_r} \right)\\
    & =\sum_{m_1\ge m_2\ge \ldots\ge m_{r}\ge 0}
    \mathcal{F}_W \left(e^{(n-1+m_1)\varepsilon_1+(n-2+m_1)\varepsilon_2+(n-3+m_2)\varepsilon_3+(n-4+m_3)\varepsilon_4+\ldots+(1+m_r)\varepsilon_{r-1}+m_r\varepsilon_r} \right)\\
\end{split}
\end{equation*}
 and 
    \begin{equation*}
    \begin{split}
    X_{\rho^\Uparrow} &= \sum_{m_1\ge m_2\ge \ldots\ge m_{r}\ge 0}\mathcal{F}_W \left(e^{\rho^\Uparrow-m_1\beta_1-\ldots - 
    {m_{r}\beta_{r}} } \right)
    \\  & = \sum_{m_1\ge m_2\ge \ldots\ge m_r\ge 0}\mathcal{F}_W \left(e^{ (1+m_1)\varepsilon_1+m_1\varepsilon_2+(1+m_2)\varepsilon_3+ m_2\varepsilon_4+\ldots+(1+m_r)\varepsilon_{r-1}+m_r\varepsilon_r } \right).
    \\    
    \end{split}
    \end{equation*}

Note that $$\mathcal{F}_W \left(e^{ (1+m_1)\varepsilon_1+m_1\varepsilon_2+(1+m_2)\varepsilon_3+ m_2\varepsilon_4+\ldots+(1+m_r)\varepsilon_{r-1}+m_r\varepsilon_r }\right) $$ is nonzero only if $m_1,\ldots,m_r,m_r+1,\ldots,m_r+1$  are distinct. 
Since $m_1\ge m_2\ge\ldots\ge m_r\ge 0$, we get that $m_r\ge0$, $m_{r-1}\ge 2$, $m_{r-2}\ge 4, \ldots, m_1\ge 2r-1$.
Thus all the nonzero terms in 
${X_\rho}$ 
and 
${X_{\rho^\Uparrow}}$ 
are the same and we get the equality.

Now 
 \begin{equation*}
\begin{split}
    X_{\rho^\Uparrow}  & = 
    \sum_{
    {m_1> m_2> \ldots> m_r> 0}
    }\mathcal{F}_W\left(e^{ (1+m_1)\varepsilon_1+m_1\varepsilon_2+(1+m_2)\varepsilon_3+ m_2\varepsilon_4+\ldots+(1+m_r)\varepsilon_{r-1}+m_r\varepsilon_r } \right)
\\    & = \frac{1}{r!}\sum_{m_1\ne m_2\ne \ldots\ne m_r\ne 0}\mathcal{F}_W \left(e^{ (1+m_1)\varepsilon_1+m_1\varepsilon_2+(1+m_2)\varepsilon_3+ m_2\varepsilon_4+\ldots+(1+m_r)\varepsilon_{r-1}+m_r\varepsilon_r } \right)
\\        & =  \frac{1}{r!}\sum_{m_1, m_2, \ldots, m_r\ge 0}\mathcal{F}_W \left(e^{ (1+m_1)\varepsilon_1+m_1\varepsilon_2+(1+m_2)\varepsilon_3+ m_2\varepsilon_4+\ldots+(1+m_l)\varepsilon_{r-1}+m_r\varepsilon_r } \right)
\\     & =  \frac{1}{r!}\mathcal{F}_W \left( \frac{e^{\rho^\Uparrow}}{
(1-e^{-\beta_1})(1-e^{-\beta_2})\cdots 
(1-e^{-\beta_r})
} \right),
\\    
\end{split}
\end{equation*}
and the claim follows from  \eqref{from product to sums of betas form} and Theorem~\ref{th1}.
\end{proof}


\section{Thick denominator identity for $\mathfrak{p}(n)$}\label{sec thick}

 In this section, we present denominator identities for the Borel subalgebra $\mathfrak{b}^{\thick}$ of $\fp(n)$, namely when 
    $\Delta_{\bar{1}}^{+}=\Delta(\mathfrak{g}_{1})$. 
 In this case, $\Delta_{\bar{1}}^{+}=\Delta_{1}$, so 
 $\rho_{\bar{1}}= \frac{n}{2}\sum_{i=1}^{n} \varepsilon_i $ and
    $
    \rho=\rho_n 
    = -\sum_{i=1}^{n} i\varepsilon_i 
    $. Let $R=\frac{R_0}{R_{1}}$, where $R_{1}=\prod_{1\le i\leq j\le n}\left(1-e^{-(\varepsilon_i+\varepsilon_j)}\right)$.
        
        We have the following theorem. 
\begin{theorem} 
\label{thm: thick}
Let $\mathfrak g=\fp(n)$ and $\Delta_{\bar{1}}^{+}=\Delta(\mathfrak{g}_{1})$. 
	Then  
\[ 
e^{\rho} R
= \sum_{w\in W} (\sgn w) w
 \left(  
\frac{e^{\rho}}{(1-e^{-\beta_1})(1-e^{-\beta_1-\beta_2})\cdots (1-e^{-\beta_1-\ldots -\beta_{n}})} 
\right),
\] 
where 
$\beta_1 = 2\varepsilon_n,\ 
\beta_2 = 2\varepsilon_{n-1},   
\ldots, 
\beta_{n} =2\varepsilon_1$.

\end{theorem}

 	\begin{proof}
We prove the identity by induction on $n$. For $n=1$, the Weyl group $W$ consists of the identity element and the only root is $\beta_1=2\varepsilon_1$. Thus the identity is evident. Suppose that the identity holds for $\mathfrak{p}(n-1)$.

	Fix the obvious root embedding of 
 	$\mathfrak{p}(n-1)$ 
 	in 
 	$\mathfrak{p}(n)$ for which
 	$\mathfrak{h}_{\mathfrak{p}(n-1)}^{*}=\text{span}\left\{ \varepsilon_{2}, 
 		\ldots,  \varepsilon_{n}\right\} $.
	 For this embedding, $\rho_{n-1}= -\sum_{i=1}^{n-1} i\varepsilon_{i+1}$ and
 	$\rho_{n}=\rho_{n-1}-\varepsilon_{1}-\ldots-\varepsilon_{n}$.	Then
	\begin{align*}
	\text{RHS} & =\mathcal{F}_{W}\left(\frac{e^{\rho_{n}}}{\left(1-e^{-2\varepsilon_{n}}\right)
	\left(1-e^{-2\varepsilon_{n-1}-2\varepsilon_{n}}\right)\cdots\left(
	1-e^{-2\varepsilon_{1}-\ldots-2\varepsilon_{n}}\right)}\right)\\
 	& =\frac{e^{-\varepsilon_{1}-\ldots-\varepsilon_{n}}}{1-e^{ 
 	-2\varepsilon_{1}-\ldots-2\varepsilon_{n}}}\mathcal{F}_{S_{n}}\left(
 	\frac{e^{\rho_{n-1}}}{\left(1-e^{-2\varepsilon_{n}}\right)\left(
 	1-e^{-2\varepsilon_{n-1}-2\varepsilon_{n}}\right)\cdots\left(
 	1-e^{-2\varepsilon_{2}-\ldots-2\varepsilon_{n}}\right)}\right)\\
 	& =\frac{e^{-\varepsilon_{1}-\ldots-\varepsilon_{n}}}{1-e^{
 	-2\varepsilon_{1}-\ldots-2\varepsilon_{n}}}\mathcal{F}_{S_{n}/S_{n-1}} 
 	\mathcal{F}_{S_{n-1}}\left(\frac{e^{\rho_{n-1}}}{\left(1-e^{-2\varepsilon_{n}}\right) 
 	\left(1-e^{-2\varepsilon_{n-1}-2\varepsilon_{n}}\right)\cdots\left(
 	1 - e^{-2\varepsilon_{2}-\ldots-2\varepsilon_{n}}\right)}\right)\\
 	& \stackrel{\text{induction}}{=}\frac{e^{-\varepsilon_{1}-\ldots-\varepsilon_{n}}}{
 	1 - e^{-2\varepsilon_{1}-\ldots-2\varepsilon_{n}}}\mathcal{F}_{S_{n}/S_{n-1}}
 	\left(e^{\rho_{n-1}}R_{n-1}\right)\\
 	& =\frac{1}{e^{\varepsilon_{1}+\ldots+\varepsilon_{n}}-e^{ 
 	-\varepsilon_{1}-\ldots-\varepsilon_{n}}}\mathcal{F}_{S_{n}/S_{n-1}}\left(
 	e^{\rho_{n-1}}R_{n-1}\right),
	\end{align*}
	where $S_{n}/S_{n-1}$ denotes a set of left coset representatives. 
	Thus the theorem is equivalent to 
	\begin{equation} \label{eqn ind}
	\left(e^{\varepsilon_{1}+\ldots+\varepsilon_{n}}-e^{
	-\varepsilon_{1}-\ldots-\varepsilon_{n}}\right)e^{\rho_{n}}R_{n} 
	= \mathcal{F}_{S_{n}/S_{n-1}}\left(e^{\rho_{n-1}}R_{n-1}\right).
	\end{equation} 

To translate this identity to be an identity of finite expressions (and not rational functions), we multiply both sides of \eqref{eqn ind} by	$R_{1,n}=R_{1,n-1} 
	\prod_{i=1}^{n}\left(1-e^{-\varepsilon_{1}-\varepsilon_{i}}\right) $, 
 	which is 
  	$W$-invariant, and we get  
	\[
	\left(e^{\varepsilon_{1}+\ldots+\varepsilon_{n}}-e^{
	-\varepsilon_{1}-\ldots-\varepsilon_{n}}\right)e^{\rho_{n}}R_{0,n} 
	= \mathcal{F}_{S_{n}/S_{n-1}}\left(e^{\rho_{n-1}}R_{0,n-1}
	\prod_{i=1}^{n}\left(1-e^{-\varepsilon_{1}-\varepsilon_{i}}\right)\right).
	\]  
	By the denominator identity of $\mathfrak{sl}(n)$ and the fact that $\rho_{n,\bar 0}-\rho_n$ is $S_n$-invariant,  we have
	\[
	e^{\rho_{n}}R_{0,n}=\mathcal{F}_{S_{n}}\left(e^{\rho_{n}}\right)
	\quad 
	\text{ and }
	\quad 
	e^{\rho_{n-1}}R_{0,n}=\mathcal{F}_{S_{n-1}}\left(e^{\rho_{n-1}}\right).
	\] 
	So the identity becomes
	\begin{equation}\label{eqn nex}
	\left(e^{\varepsilon_{1}+\ldots+\varepsilon_{n}}-e^{
	-\varepsilon_{1}-\ldots-\varepsilon_{n}}\right)\mathcal{F}_{S_{n}}\left(e^{\rho_{n}}\right)
	= \mathcal{F}_{S_{n}/S_{n-1}}\left(\mathcal{F}_{S_{n-1}}\left(e^{\rho_{n-1}}\right) 
	\prod_{i=1}^{n}\left(1-e^{-\varepsilon_{1}-\varepsilon_{i}}\right)\right).
	\end{equation}
	Since the term $\prod_{i=1}^{n}\left(1-e^{-\varepsilon_{1}-\varepsilon_{i}}\right)	$ is $S_{n-1}$-invariant, the RHS of \eqref{eqn nex} equals
\begin{align*}
\mathcal{F}_{S_{n}/S_{n-1}}
&\left(\mathcal{F}_{S_{n-1}}\left(e^{\rho_{n-1}}\right)
\prod_{i=1}^{n}\left(1-e^{-\varepsilon_{1}-\varepsilon_{i}}\right)\right)  \\ 
&=\mathcal{F}_{S_{n}/S_{n-1}}\left(\mathcal{F}_{S_{n-1}}\left(e^{\rho_{n-1}}\prod_{i=1}^{n}\left(1-e^{-\varepsilon_{1}-\varepsilon_{i}}\right)\right)\right)\\
 & =\mathcal{F}_{S_{n}}\left(e^{\rho_{n-1}}\prod_{i=1}^{n}\left(1-e^{-\varepsilon_{1}-\varepsilon_{i}}\right)\right).
\end{align*}
	Hence, as $\left(e^{\varepsilon_{1}+\ldots+\varepsilon_{n}}-e^{
	-\varepsilon_{1}-\ldots-\varepsilon_{n}}\right)$ is $S_n$-invariant, \eqref{eqn nex} becomes
\begin{equation}\label{2 orbits}
	\mathcal{F}_{S_{n}}\left(e^{\rho_{n}}\left(e^{\varepsilon_{1}+\ldots+\varepsilon_{n}}-e^{
	-\varepsilon_{1}-\ldots-\varepsilon_{n}}\right)\right)
	= \mathcal{F}_{ S_{n}}\left(e^{\rho_{n-1}} 
	\prod_{i=1}^{n}\left(1-e^{-\varepsilon_{1}-\varepsilon_{i}}\right)\right).
	\end{equation}

Finally, we are left to prove an equality between two $S_n$-anti-invariant finite expressions, and
by Lemma~\ref{orbit}, we are reduced to studying regular elements.
The LHS of \eqref{2 orbits} has  two $S_n$-orbits, which correspond to
  	$\mathcal{F}_{S_n}\left( e^{\rho_{n}+\varepsilon_{1}+\ldots+\varepsilon_{n}}\right)$ 
 	and 
 	$\mathcal{F}_{S_n}\left(-e^{\rho_{n}-\varepsilon_{1}-\ldots-\varepsilon_{n}}\right)$.

  By expanding the inside of the RHS of 
  \eqref{2 orbits}  
  we obtain
 \begin{equation}\label{eqn orb}
 e^{\rho_{n-1}} 
	\prod_{i=1}^{n} \left(1-e^{-\varepsilon_{1}-\varepsilon_{i}}\right) 
	=\sum_{A\subset\{\varepsilon_1+\varepsilon_i\mid i=1,\ldots,n  \}} a_{\lambda_A} e^{\lambda_A}
,
	\end{equation}
where $\lambda_{ A}=\rho_{n-1}-\sum_{\alpha\in A}\alpha$ and $a_{\lambda_A}=(-1)^{|A|}$.

 	 	If $A$ is empty, then $\lambda_{ A}=\rho_{n-1}=\rho_n+\varepsilon_1+\ldots+\varepsilon_n$ is regular and $a_{\lambda_{ A}}=1$. If $A$ is the entire set, then   
 	$\lambda_{ A}=-2\varepsilon_2-3\varepsilon_3-\ldots-n\varepsilon_n-(n+1)\varepsilon_1$ is regular and $a_{\lambda_{ A}}=(-1)^n$. In the latter case, $\lambda_{ A}=y(\rho_{n}-\varepsilon_{1}-\ldots-\varepsilon_{n})$ where $y$ is the permutation $(12\ldots n)$ and  $\sgn y=(-1)^{n-1}$. 
 	We claim that if $n$ is odd then these are the only two regular elements in the RHS of \eqref{2 orbits}, while if $n$ is even then we have two more regular elements that cancel each other in the sum. This will imply that \eqref{eqn orb} holds as required.
 	
 	Now suppose $A\subsetneq \{\varepsilon_1+\varepsilon_i\mid i=1,\ldots,n  \}$ and $A$ is nonempty. Write
 	$$\lambda_A=\sum_{i=1}^{n} b_i\varepsilon_i.$$
 	 Then the coefficients $b_1,\ldots,b_n$ are contained in $\{1,2,\ldots,n\}$, and they are distinct since $\lambda_A$ is assumed to be regular. Moreover, for $k\geq 2$, either $b_k=k-1$ or $b_k=k$.
 	 Let $k\geq 2$ be the smallest integer for which $b_k=k$. Then $b_i=i$ for all $i\geq k$, since  $b_1,\ldots,b_n$ are distinct. It follows that $k>2$ since we assume that $A$ is not the entire set and that $\lambda_A$ is regular. We have two possibilities:  $A=\{\varepsilon_1+\varepsilon_{k},\ldots,\varepsilon_1+\varepsilon_{n}  \}$ and $A'=\{\varepsilon_1+\varepsilon_{k},\ldots,\varepsilon_1+\varepsilon_{n},2\varepsilon_1  \}$. 
 	 Regularity  implies that $k=\frac{n}{2}+1$ in the former case, while $k=\frac{n}{2}+2$ in the latter case. Clearly, this implies that $n$ is even. Finally, since $\lambda_A$ differs from $\lambda_{A'}$ only by the transposition $(1 \ \frac{n}{2})$, we conclude that these two additional regular elements cancel each other in the sum.
 		\end{proof} 



\section{Some remarks}\label{sec remarks}

\subsection{The character version of the denominator identity}\label{sec char}
The denominator identities written in this paper are given in terms of supercharacters. One can translate them into characters. In this case $R_0$ stays the same,  
$R_1=\prod_{\alpha\in\Delta(\mathfrak{g}_1)}\left(1+e^{-\alpha}\right)$, the identity in Theorem~\ref{th1} takes the form

\begin{equation*}
	e^{\rho}R=\frac{1}{r !}
	     \sum_{w\in W} (\sgn w) w\left(\frac{e^{\rho^{\Uparrow}} }{
	\prod_{\beta\in S} \left(1+ e^{-\beta} \right)} 
	   \right),
	\end{equation*}
	and the identities in Theorem~\ref{thm thin sum of betas} and Theorem~\ref{thm: thick}  take the form
\[
e^\rho R=\mathcal{F}_W 
\left( 
\frac{e^{\rho}}{(1+e^{-\beta_1})(1-e^{-\beta_1-\beta_2})\cdots (1+(-1)^r e^{-\beta_1-\ldots -\beta_r})} 
\right)
\]
for the appropriate choices of $\beta_1,\ldots,\beta_r$.

\subsection{Representation-theoretical meaning of the denominator identity}  \label{complex}

It would be interesting to find a complex of thin Kac modules (or thick Kac modules) whose Euler characteristic yields the denominator identity.

For a given dominant integral weight $\lambda$, we let $V(\lambda)$ denote the simple $\mathfrak{g}_0$-module with highest weight $\lambda$ 
 with respect to the fixed Borel $\mathfrak{b}_0$ of $\mathfrak{g}_0$. The thin Kac module corresponding to $\lambda$ is defined to be 
$
\nabla(\lambda) := \Coind_{\mathfrak{g}_0\oplus \mathfrak{g}_1}^{\mathfrak{g}} V(\lambda),
$ 
where we take the parity of  the superspace $V(\lambda)$ to be purely even or odd according the sign convention used in \cite[Section 2.3]{IRS19}, and denote this parity by  $\sgn \lambda$.
Then the supercharacter of $\nabla(\lambda)$ is $$\sch\nabla(\lambda)=(\sgn \lambda)\frac{R_{-1}}{e^{\rho}R_0}\cdot\mathcal F_W\left(e^{\lambda+\rho}\right)$$  (see \cite[Lemma 2.4.1]{IRS19}).
After substitution, the formula in Theorem~\ref{thm thin sum of betas} takes the form 
$$\sch L(0)=\sum_{ i_1\ge\ldots\ge i_r\ge 0}(-1)^{i_1+\ldots+i_r}\sch\nabla\left( -i_1\beta_1-\ldots-i_r\beta_r \right).$$

For $r=2$, we conjecture that the complex has the following form.
\xymatrixrowsep{.8cm}
\xymatrixcolsep{.8cm}
\[  
\xymatrix@-1.9pc{
\ddots & & & & & & & & & & & & & & &\\ 
\ldots &  \nabla(-n\beta_1 -n\beta_2)  & & & & & & & & & & & & & &\\  
& & & & & & & & & & & & & & & \\ 
& & & & & & & & & & & & & & & \\ 
\ldots & \vdots \ar[uuu]& & \ddots & & & & & & & & & & & &\\ 
& & & & & & & & & & & & & & & \\ 
& & & & & & & & & & & & & & & \\ 
\ldots &  \nabla(-n\beta_1 -2 \beta_2)\ar[uuu]   & &  
 \ldots\ar[ll]  & & & \nabla(-2\beta_1 -2 \beta_2)\ar[lll]  & 
& & & & & & & & \\ 
& & & & & & & & & & & & & & & \\ 
& & & & & & & & & & & & & & & \\ 
\ldots &  \nabla(-n\beta_1 -\beta_2)\ar[uuu]   & & 
 \ldots \ar[ll] & & & 
\nabla(-2\beta_1 -\beta_2)\ar[uuu]\ar[lll] & 
& {\nabla}(-\beta_1-\beta_2)\ar[ll]
& &  &  & & & & \\ 
& & & & & & & & & & & & & & & \\ 
& & & & & & & & & & & & & & & \\ 
\ldots &  \nabla(-n\beta_1)\ar[uuu]   & & \ldots \ar[ll]
& & & {\nabla}(-2 \beta_1)\ar[uuu]\ar[lll]  & 
& {\nabla}(-\beta_1)\ar[ll]\ar[uuu]&  {\nabla}(0) \ar[l]& & &   L(0)\ar[lll] & & & 0.\ar[lll] \\ 
}
\]

\bibliography{branching}   
 
\end{document}